\documentclass[11pt,a4paper]{article}

\usepackage[T1]{fontenc}
\usepackage[utf8]{inputenc} % à retirer si compilation LuaLaTeX/XeLaTeX

\usepackage{lmodern}

\usepackage{geometry}
\usepackage{graphicx}
\usepackage{amsmath,amssymb,amsthm,mathtools,stmaryrd,latexsym}
\usepackage{enumitem}
\usepackage[table,dvipsnames]{xcolor}
\usepackage{hyperref}

\newcommand{\bb}[1]{\mathbb{#1}}

\newcommand{\Th}{\mathcal{T}_{h}}
\newcommand{\ThGamma}{\mathcal{T}_{h}^\Gamma}

\newcommand{\vertiii}[1]{%
  {\left\vert\kern-0.25ex\left\vert\kern-0.25ex\left\vert #1
  \right\vert\kern-0.25ex\right\vert\kern-0.25ex\right\vert}
}

\newtheorem{assumption}{Assumption}
\newtheorem{theorem}{Theorem}
\newtheorem{lemma}{Lemma}
\newtheorem{proposition}{Proposition}
\theoremstyle{definition}

\theoremstyle{remark}
\newtheorem{remark}{Remark}

\title{A penalized $\varphi$-FEM scheme for the Poisson Dirichlet problem}

\author{
Raphaël Bulle\thanks{Université de Strasbourg, CNRS, Inria, ICube, F-67000 Strasbourg, France. \texttt{raphael.bulle@inria.fr}}
\and
Michel Duprez\thanks{Université de Strasbourg, CNRS, Inria, ICube, F-67000 Strasbourg, France. \texttt{michel.duprez@inria.fr}}
\and
Vanessa Lleras\thanks{IMAG, Univ Montpellier, CNRS UMR 5149, 34090 Montpellier, France. \texttt{vanessa.lleras@umontpellier.fr}}
\and
Killian Vuillemot\thanks{IMAG, Univ Montpellier, CNRS UMR 5149, 34090 Montpellier, France. \texttt{killian.vuillemot@umontpellier.fr}}
}

\date{} % laisser vide

\begin{document}

\maketitle

\begin{abstract}
In this work, we analyze a penalized variant of the $\varphi$-FEM scheme for the Poisson equation with Dirichlet boundary conditions.
The $\varphi$-FEM is a recently introduced unfitted finite element method based on a level-set description of the geometry, which avoids the need for boundary-fitted meshes.
Unlike the original $\varphi$-FEM formulation, the method proposed here enforces boundary conditions through a penalization term.
This approach has the advantage that the level-set function is required only on the cells adjacent to the boundary in the variational formulation.
The scheme is stabilized using a ghost penalty technique.
We derive a priori error estimates, showing optimal convergence in the $H^1$ semi-norm and quasi-optimal convergence in the $L^2$ norm under suitable regularity assumptions.
Numerical experiments are presented to validate the theoretical results and to compare the proposed method with both the original $\varphi$-FEM and the standard fitted finite element method.
\end{abstract}


% -------------------------------------------------
% Début du papier
% -------------------------------------------------

\section{Introduction}

The finite element method (FEM) usually requires meshes that conform to the domain boundary, which can be restrictive for complex or evolving geometries. To address this limitation, Immersed Boundary Methods (IBM) \cite{IBMrev} allow the use of meshes that do not fit the physical boundary by extending the problem to a larger domain.  More recently, CutFEM \cite{cutfem, cutfem_mixed} has been proposed as a robust alternative, based on integrating the weak formulation only over the parts of elements cut by the boundary and adding stabilization terms to ensure numerical stability and coercivity. 
Another alternative, called the Shifted Boundary Method \cite{sbm}, uses a Taylor expansion of the boundary condition.
The $\varphi$-FEM paradigm introduced in \cite{phifem} is based on a level-set description of the domain, where the solution is written as 
$u=\varphi w$, automatically enforcing boundary conditions and achieving optimal accuracy and good conditioning without requiring boundary or cut-cell integration. This method has been extended to various PDEs and interface problems \cite{ cotin2023phi,phiFEM2, duprez2023phi, duprez2022immersed}.
More recently, this approach has been adapted to a finite difference scheme \cite{phiFD}, combined with a neural network \cite{phifem_fno} and an a posteriori estimator, and has been analysed to estimate the error \cite{becker2025residual}.

In this paper, we consider the Poisson equation with Dirichlet boundary conditions,
 \begin{equation}
      \label{eq:main_poisson_equation_dirichlet_homo}
      \begin{cases}
            - \Delta u & = f\,, \ \hfill \text{ in } \Omega\,, \\
            \hfill u   & = 0\,, \  \hfill \text{ on } \Gamma\,,
      \end{cases}
\end{equation} 
The present contribution is devoted to the numerical analysis of a $\varphi$-FEM scheme introduced in \cite{cotin2023phi} for this problem (see Section 2). This $\varphi$-FEM scheme enforces boundary conditions weakly by introducing a penalization on the unfitted boundary while still using a level-set description of the domain. So this method combines the robustness of mixed formulations with the flexibility of immersed finite element methods.
We define $\Omega \subset \bb{R}^d$ (here $d=2, 3$) a domain of boundary $\Gamma$, $f \in L^2(\Omega)$ and $n$ the outward unit normal to $\Omega$. We assume that the domain and its boundary are given by a level-set function $\varphi$ such that 
\begin{equation}\label{eq:def_omega_gamma_by_phi}
      \Omega:= \{ \varphi < 0 \} \qquad \text{ and } \qquad \Gamma:= \{ \varphi = 0 \}\,.
\end{equation}

The present paper is organized as follows: in Section 2, we present the scheme and the main theoretical results. In the second and third Section, we demonstrate these results. Section 4 analyzes the conditioning number of associated matrix. Finally,  Section 5 is devoted to numerical illustrations.

\section{Proposed scheme and theoretical results}

Let $\mathcal{T}_h^\mathcal{O}$ be a Cartesian grid covering the domain $\mathcal{O}$, with triangular cells of size $h$. Let $\varphi_h = I_{h, \mathcal{O}}^{(l)} \varphi$ the standard Lagrange continuous interpolation of $\varphi$ (of degree $l>0$) on $\mathcal{T}_h^\mathcal{O}$.
This interpolation is then used to construct a submesh $\mathcal{T}_h$ of $\mathcal{T}_h^\mathcal{O}$ collecting all the cells of the grid intersecting the domain $\{ \varphi_h < 0 \}$, i.e.
\begin{equation*}
%      \label{eq:main_Th}
      \mathcal{T}_h:= \left\{T \in \mathcal{T}_h^{\mathcal{O}} : \ T \cap \{ \varphi_h < 0 \} \neq \emptyset  \right\}\,.
\end{equation*}

Let us finally introduce another submesh of $\Th$, collecting the cells intersecting the boundary, i.e. $\{ \varphi_h = 0 \}$
and a set containing the facets of this submesh, given by 
\begin{equation}
      \label{eq:main_ThGamma}
      \mathcal{T}_h^\Gamma:= \left\{T \in \mathcal{T}_h \: \ T \cap \{ \varphi_h = 0 \} \neq \emptyset  \right\}\,,\text{ and }
      \mathcal{F}_h^\Gamma:= \{ F \in \mathcal{T}_h^\Gamma \setminus \partial\Omega_h  \} \,.
\end{equation}

Let us now consider the finite element spaces: for $k\geqslant1$,

\begin{equation}\label{eq:main_Vhk_phifem}
      V_h^{(k)}:= \{ v_h \in H^1(\Omega_h) \: \ v_h|_T \in \bb{P}_k(T) \ \forall \ T \in \mathcal{T}_h \} \,,
\end{equation}
and a second one, the local version of the previous space, on $\ThGamma$, 
\begin{equation*}
      Q_{h}^{(k)}:= \{ q_h \in L^2(\Omega_h^{\Gamma}) \: \ v_h|_T \in \bb{P}_k(T) \ \forall \ T \in \mathcal{T}_h^{\Gamma} \} \,.
\end{equation*}

We will now present the $\varphi$-FEM penalized scheme analyzed here, that relies on the $\varphi$-FEM paradigm used only in the cells of $\Omega_h^\Gamma$. For this scheme, we thus introduce an auxiliary variable $p$ on $\Omega_h^\Gamma$ such that $u = \varphi p$; equation that will be enforced by penalization in the following scheme. 

The penalized $\varphi$-FEM scheme to solve \eqref{eq:main_poisson_equation_dirichlet_homo} is given by: find $u_h \in V_h^{(k)}$ and $p_{h} \in Q_h^{(k)}$ such that 
\begin{multline}\label{eq:scheme_poisson_dirichlet_phi_fem_dual}
    \int_{\Omega_h} \nabla u_h \cdot \nabla v_h - \int_{\partial\Omega_h} \frac{\partial u_h}{\partial n} v_h
    + \frac{\gamma}{h^2} \int_{\Omega_h^{\Gamma}} (u_h - \frac{1}{h}\varphi_h p_{h}) (v_h - \frac{1}{h}\varphi_h q_{h}) \\
    + G_h^{lhs}(u_h,v_h)
    = \int_{\Omega_h} f v_h + G_h^{rhs}(v_h)\,,
\end{multline}
for all $v_h \in V_h^{(k)}$, $q_{h} \in Q_h^{(k)}$ with
\begin{equation}\label{eq:main_stabs_lhs_poisson}
      G_h^{lhs}(u,v) = \sigma_D h \sum_{E \in \mathcal{F}_h^\Gamma} \int_E \left[ \frac{\partial u}{\partial n} \right]\left[ \frac{\partial v}{\partial n} \right]  + \sigma_D h^2 \sum_{T \in \mathcal{T}_h^\Gamma}\int_T \Delta u \Delta v \,,
\end{equation}
and 
\begin{equation}\label{eq:main_stabs_rhs_poisson}
      G_h^{rhs}(v) = - \sigma_D h^2 \sum_{T \in \mathcal{T}_h^\Gamma}\int_T f \Delta v \,
\end{equation}
for %$f_h=I_hf$ and 
some $\gamma,~\sigma_D>0$.
As said in the introduction, this scheme has been introduced in  \cite{cotin2023phi}, but not analyzed.

\begin{remark}
We can remark three important points: 
\begin{itemize}[label=\textbullet]
    \item In the case of $\bb{P}^1$ finite elements, the second-order terms (the ones multiplied by $\sigma_D h^2$) are not necessary. 
    \item In the case of non-homogeneous Dirichlet boundary conditions $u=u_D$, all you need to do is the modification of the penalization form, given by $u = \varphi p + u_D$. It then remains to adapt the scheme to this equation. 
    \item A final observation concerns the comparison with the direct $\varphi$-FEM scheme \cite{phifem}, relying on the assumption $u = \varphi w$ over the entire mesh, thereby redefining the problem’s unknown throughout the domain. This has two major consequences: 
    \begin{itemize}
        \item the computation relies on the level-set only near the boundary, which makes the scheme more robust to the level-set and, in particular, does not require the gradient or the second derivatives of the level-set function, which could present a singularity (for example, if we consider the signed distance function). 
        \item the scheme is compatible with the Neumann approach presented in \cite{phiFEM2}. 
    \end{itemize}
\end{itemize}
\end{remark}

Let us now recall some important hypothesis on the mesh and lemmas from \cite{phifem} and \cite{phifemstokes} and introduce our main theoretical results. 

\begin{assumption}
    \label{assumption1}
    The boundary $\Gamma$ can be covered by open sets $\mathcal{O}_i$, $i=1, \dots, I$ on which we can introduce local coordinates $\xi_1,\dots,\xi_d$ with $\xi_d =\varphi$ such that, up to the order $k+1$, all the partial derivatives $\partial^\alpha \xi_i / \partial x^\alpha$ and $ \partial x^\alpha / \partial^\alpha \xi_i $ are bound by a constant $C_0 > 0$. Moreover, on $\mathcal{O}$, $\varphi$ is of class $C^{k+1}$ and there exists $m>0$ such that $|\varphi| \geqslant m$ on $\mathcal{O} \setminus \cup_{i=1,\dots,I} \mathcal{O}_i$.
\end{assumption}

\begin{assumption}\label{assumption2}
    The approximated boundary, defined by $\Gamma_h = \{\varphi_h = 0\}$, can be covered by patches of elements $\{ \Pi_r\}_{r=1,\dots,N_\Pi}$ such that:
    \begin{itemize}
        \item Each patch $\Pi_r$ can be written as $\Pi_r=\Pi_r^\Gamma\cup T_r$ where $\Pi_r^\Gamma \subset \Th^\Gamma$ and  $T_r \in \Th \setminus \Th^\Gamma$. Moreover $\Pi_r$,  contains at most $M$ connected elements, with $M$ independent of $h$ ;
        \item The mesh $\Th^\Gamma$ satisfies $\ThGamma = \cup_{r=1,\dots, N_\Pi} \Pi_r^\Gamma$;
        \item Two patches $\Pi_r$ and $\Pi_s$ are disjoints if $r \neq s$.
    \end{itemize}
\end{assumption}

Let us now present the main theorem of this paper: 
\begin{theorem}\label{thm:convergence_dual_phi_fem}
    We suppose that assumptions \ref{assumption1} and \ref{assumption2} on $\Gamma$ and $\mathcal{T}_h$ 
    are satisfied, $f \in H^{k-1}(\Omega_h)$, $\Omega \subset \Omega_h$ and $\gamma,~\sigma_D$ large enough.
    Consider  $u \in {H^{k+1}(\Omega)}$ the solution of \eqref{eq:main_poisson_equation_dirichlet_homo} and $u_h \in V_h$ the solution of the discrete scheme \eqref{eq:scheme_poisson_dirichlet_phi_fem_dual}. 
    It holds 
    \begin{equation*}
        \label{eq:estimate_phi_fem_dual_theorem}
        | u - u_h |_{1, {\Omega}} \leqslant C h^k {\| f \|_{k-1, \Omega_h}} 
    \end{equation*}
    with $C>0$ a constant independent on $h$.
\end{theorem}

We note that our analysis requires less regularity on \(u\) compared to \cite{phifem}. Moreover, the results of \cite{phifem} can be adapted to show that the condition number of the associated matrix scales like \(O(h^{-2})\), while the \(L^2\)-error is of order \(O(h^{k+1/2})\). In the final section of this paper, optimal convergence of order \(O(h^{k+1})\) is observed numerically.

\section{Coercivity of the bilinear form}

This section is devoted to the proof of the coercivity needed for Theorem \ref{thm:convergence_dual_phi_fem}.

\begin{proposition}\label{lemma:coercivity_dual_phi_fem}
    Consider the bilinear form given by
    \begin{multline}
        \label{eq:bilinear_form_poisson_dual}
        a_h(u,p; v,q) = \int_{\Omega_h} \nabla u \cdot \nabla v - \int_{\partial\Omega_h} \frac{\partial u}{\partial n} v
        + \frac{\gamma}{h^2} \int_{\Omega_h^{\Gamma}} (u - \frac{1}{h}\varphi_h p) (v - \frac{1}{h}\varphi_h q) \\
        + \sigma_D h \sum_{F \in \mathcal{F}_h^{\Gamma}} \int_F \left[ \frac{\partial u}{\partial n}  \right]  \left[ \frac{\partial	v}{\partial n} \right]
        + \sigma_D h^2 \int_{\Omega_h^{\Gamma}}  \Delta u \Delta v \,.
    \end{multline}
  Under assumptions 
    1 and 2, for $\gamma$ and $\sigma_D$ sufficiently big, $a_h$
    is coercive on $V_h^{(k)}\times Q_h^{(k)}$ for the norm
    \begin{equation}
        \label{eq:norm_coercitiy_dual_poisson_phi_fem}
        \vertiii{(u,p)}^2_h = | u |_{1,\Omega_h}^2 + \frac{1}{h^2} \left\| u- \frac{1}{h} \varphi_h p \right\|^2_{0,\Omega_h^\Gamma} + h \sum_{F \in \mathcal{F}_h^{\Gamma}}\left\| \left[ \frac{\partial u}{\partial n}  \right] \right\|^2_{0, F}
        + h^2 \| \Delta u \|^2_{0, \Omega_h^{\Gamma}}\,,
    \end{equation}
    i.e. $\vertiii{(u_h,p_h)}^2_h\geqslant  C a_h(u_h,p_h; u_h,p_h)$ for each $(u_h,p_h)\in V_h^{(k)}\times Q_h^{(k)}$.
\end{proposition}

Before proving Proposition \ref{lemma:coercivity_dual_phi_fem}, we need the following Lemma : 
\begin{lemma}\label{lemma:phi_p}
For all $u_h \in V_h^{(k)}$ and $p_h \in Q_h^{(k)}$, there exists $C>0$ such that  \begin{equation*}%\label{eq:lamma_phi_p}
        \| \varphi_h p_h \|_{0,\Omega_h^\Gamma} \leqslant C (h \|\nabla u_h \|_{0,\Omega_h^\Gamma} + \| u_h - \varphi_h p_h\|_{0,\Omega_h^\Gamma})\,.
    \end{equation*}
\end{lemma}

\begin{proof}
    Using the Poincaré inequality, combined to the triangular inequality and an inverse inequality, we obtain
    \begin{multline*}
       \| \varphi_h p_h \|_{0,\Omega_h^\Gamma} \leqslant Ch\| \nabla \varphi_h p_h \|_{0,\Omega_h^\Gamma} 
        \leqslant Ch(\| \nabla u_h \|_{0,\Omega_h^\Gamma} + \| \nabla (u_h -\varphi_h p_h) \|_{0,\Omega_h^\Gamma} )       \\
                       \leqslant C(h\| \nabla u_h \|_{0,\Omega_h^\Gamma} +  \| u_h - \varphi_h p_h \|_{0,\Omega_h^\Gamma})\,.
    \end{multline*}
\end{proof}

 We also need the following Lemma :
 \begin{lemma}[{cf. \cite[Lemma 3.3]{phifem}}]\label{lemma:magic_dirichlet}
Under the Assumption \ref{assumption2}, for all $\beta >0$, % and $s \in \bb{N}^*$, 
one can choose $\alpha \in ]0,1[$ depending only on the regularity of the mesh, % and of $s$, 
such that for all $v_h \in V_h^{(k)}$
    \begin{equation*}%\label{eq:magic_dirichlet}
        |v_h|^2_{1,\Omega_h^{\Gamma}} \leqslant \alpha |v_h|_{1,\Omega_h}^2 + \beta h
\sum_{F \in \mathcal{F}_h^{\Gamma}}     \left\| \left[ \frac{\partial v_h}{\partial n} \right] \right\|_{0, F}^2+ \beta h^2  \| \Delta v_h \|_{0,\Omega_h^{\Gamma}}^2 \,.
    \end{equation*} 
\end{lemma}

\begin{proof}[{Proof of Proposition \ref{lemma:coercivity_dual_phi_fem}}]
Let $B_h$ the strip between $\partial\Omega_h$ and $\Gamma_h$, i.e. $B_h = \{ \varphi_h > 0 \} \cap \Omega_h$. Since $\varphi_h = 0$ on $\Gamma_h$, the boundary term of \eqref{eq:bilinear_form_poisson_dual}
can be written for all $(u_h,p_h)\in V_h^{(k)}\times Q_h^{(k)}$ as :
        \begin{multline}
        \label{eq:coercivity_first_eq}
        \int_{\partial \Omega_h} \frac{\partial u_h}{\partial n} u_h
        = \overbrace{\int_{B_h} |\nabla u_h|^2}^{I}- \overbrace{\int_{\Gamma_h} \frac{\partial u_h}{\partial n} (u_h-\frac{1}{h}\varphi_h p_h)}^{II}\\ 
        -  \underbrace{\sum_{F \in \mathcal{F}_h^\Gamma} \int_{F\cap B_h} \left[\frac{\partial u_h}{\partial n}\right] u_h }_{III}
         + \underbrace{\int_{B_h} \Delta(u_h)u_h}_{IV}.
    \end{multline}
 Using Lemma \ref{lemma:magic_dirichlet}, for all $\beta>0$, there exists $\alpha\in(0,1)$ such that
    \begin{equation*}
        I \leqslant
         |u_h|_{1, \Omega_h^\Gamma}^2
\leqslant        
         \alpha |u_h |_{1, \Omega_h}^2 + \beta h\sum_{F \in \mathcal{F}_h^{\Gamma}}  \left\| \left[ \frac{\partial u_h}{\partial n} \right] \right\|_{0, F}^2 + \beta h^2 \| \Delta u_h \|_{0, \Omega_h^\Gamma}^2\,.
    \end{equation*}
 Moreover, thanks to the trace inequality and an inverse inequality, for all $\varepsilon>0$,
    \begin{multline*}
        II  \leqslant C \left( \frac{1}{\sqrt{h}} \|\nabla u_h \|_{0, \Omega_h^\Gamma} + \sqrt{h} |\nabla u_h |_{1, \Omega_h^\Gamma} \right) \\
        \times\left( \frac{1}{\sqrt{h}} \left\| u_h- \frac{1}{h} \varphi_h p_h \right\|_{0, \Omega_h^\Gamma} + \sqrt{h} \left| u_h- \frac{1}{h} \varphi_h p_h \right|_{1, \Omega_h^\Gamma} \right) \\ 
             \leqslant \frac{C}{h} | u_h |_{1, \Omega_h^\Gamma} \left\| u_h- \frac{1}{h} \varphi_h p_h \right\|_{0, \Omega_h^\Gamma} 
             \leqslant C \varepsilon | u_h |_{1, \Omega_h^\Gamma}^2 + \frac{C}{\varepsilon h^2} \left\| u_h- \frac{1}{h} \varphi_h p_h \right\|_{0, \Omega_h^\Gamma}^2\,.
    \end{multline*}
 Cauchy-Schwarz, Young and trace inequalities as well Lemma \ref{lemma:phi_p} lead to
    \begin{equation*}
        III \leqslant \frac{C}{\varepsilon} h \sum_{F \in \mathcal{F}_h^{\Gamma}} \left\| \left[ \frac{\partial u_h}{\partial n} \right] \right\|_{0, F}^2 + C \varepsilon | u_h |_{1, \Omega_h^\Gamma}^2 + \frac{C \varepsilon}{h^2}\left\| u_h- \frac{1}{h} \varphi_h p_h \right\|_{0, \Omega_h^\Gamma}^2\,.
    \end{equation*}
    Finally, Lemma \ref{lemma:phi_p} gives
    \begin{equation*}
        IV \leqslant \frac{C h^2}{\varepsilon} \| \Delta u_h \|_{0, \Omega_h^\Gamma}^2 + C \varepsilon | u_h |_{1, \Omega_h^\Gamma}^2 + \frac{C \varepsilon}{h^2} \left\| u_h- \frac{1}{h} \varphi_h p_h \right\|^2_{0, \Omega_h^\Gamma}.
    \end{equation*}
   Hence, 
   \begin{multline*}
        a_h(u_h,p_h;u_h,p_h) \geqslant (1- \alpha - C \varepsilon) | u_h |_{1, \Omega_h}^2 \\ 
        + \left( \sigma_D - \frac{C}{\varepsilon} - \beta \right) \left( h \sum_{F \in \mathcal{F}_h^{\Gamma}}\left\| \left[ \frac{\partial u_h}{\partial n} \right] \right\|_{0, F}^2 + h^2  \| \Delta u_h \|_{0, \Omega_h^\Gamma}^2 \right) \\
        + \left(\gamma - C \varepsilon - \frac{C}{\varepsilon} \right) \frac{1}{h^2} \left\| u_h- \frac{1}{h} \varphi_h p_h \right\|^2_{0, \Omega_h^\Gamma}.
    \end{multline*}
    Choosing $\varepsilon$ sufficiently small, $\sigma_D$ and $\gamma$ sufficiently big, we obtain the desired result.
\end{proof}

\section{Proof of the $H^1$ estimate}

To prove Theorem \ref{thm:convergence_dual_phi_fem}, we will need the two following results: 
 \begin{lemma}[{cf. \cite[Lemma 3.4]{phiFEM2}}]\label{lemma:phifem_neumann_lemma_34}
     Under the assumption \ref{assumption1}, all function $v \in H^{s} ( \Omega_h)$ vanishing on $\Omega$, with $1 \leqslant s \leqslant k+1$, satisfies 
     $$\| v\|_{0, \Omega_h \setminus \Omega} \leqslant C h^s \| v \|_{s, \Omega_h \setminus \Omega}.$$
 \end{lemma}

\begin{lemma}
  \label{lemma:hardy}
  We assume that the domain $\Omega$ is given by the
  level-set $\phi$, cf. (\ref{eq:def_omega_gamma_by_phi}), and satisfies Assumption \ref{assumption1}. Then, for any $u \in H^{k + 1} (\mathcal{O})$ vanishing on $\Gamma$,
  \[ \left\| \frac{u}{\phi} \right\|_{k, \mathcal{O}} \le C \| u
     \|_{k + 1, \mathcal{O}} \]
  with $C > 0$ depending only on the constants in Assumption \ref{assumption1}. 
\end{lemma}

\begin{proof}[Proof of Theorem \ref{thm:convergence_dual_phi_fem}, $H^1$ estimate]
    Let $\tilde u  \in H^{k+1}(\Omega_h)$ an extension of $u$ from $\Omega$ to $\Omega_h$ such that $\tilde{u} = u$ on $\Omega$ and
    $$\Vert \tilde{u}\Vert_{k+1,\Omega_h}\leqslant C\Vert u\Vert_{k+1,\Omega}\leqslant C \Vert f\Vert_{k-1,\Omega}.$$
    We consider $\tilde{f} := - \Delta \tilde{u}$
     and $p := \frac{h}{\varphi} \tilde{u}$ on $\Omega_h$. 
    Then,
    \begin{multline*} a_h (\tilde{u}, p ; v_h, q_h) = \int_{\Omega_h} \tilde f v_h
        - \sigma_D h^2 \int_{\Omega_h^\Gamma} \tilde f \Delta v_h\\
        + \frac{\gamma}{h^2} \int_{\Omega_h^\Gamma} (\tilde{u} - \frac{1}{h} \varphi_h p)(v_h - \frac{1}{h} \varphi_h q_h)\,
    \end{multline*}
    for each $(v_h,q_h)\in V_h^{(k)}\times Q_h^{(k)}$.
    We obtain the following Galerkin orthogonality
    \begin{multline}\label{eq:galerkin_proof_H1}
        a_h (\tilde{u} - u_h, p - p_h ; v_h, q_h) = \int_{\Omega_h} (\tilde f -f)v_h 
        - \sigma_D h^2 \int_{\Omega_h^\Gamma} (\tilde f - f) \Delta v_h\\
        + \frac{\gamma}{h^2} \int_{\Omega_h^\Gamma} (\tilde{u} - \frac{1}{h} \varphi_h p)(v_h - \frac{1}{h} \varphi_h q_h) .
    \end{multline}
       Thus, using the coercivity of the bilinear form (c.f. Proposition \ref{lemma:coercivity_dual_phi_fem}),
    \begin{align*}
        c \vertiii{(u_h - I_h \tilde{u}, p_h - I_hp)}_h & \leqslant \sup_{(v_h, q_h)} \frac{a_h(u_h - I_h \tilde{u}, p_h - I_hp; v_h, q_h)}{\vertiii{(v_h, q_h)}_h} \\
                                                         & \leqslant \sup_{(v_h, q_h)} \frac{ I - II - III}{\vertiii{(v_h, q_h)}_h},
    \end{align*}
    with
   $$
        I    = a_h (e_u, e_p ; v_h, q_h) \,,                                                                                  \quad 
        II   = \int_{\Omega_h} (\tilde f -f)v_h
        - \sigma_D h^2 \int_{\Omega_h^\Gamma} (\tilde f - f) \Delta v_h\,,                                                          $$
        $$  III  = \frac{\gamma}{h^2} \int_{\Omega_h^\Gamma} (\tilde{u} - \frac{1}{h} \varphi_h p)(v_h - \frac{1}{h} \varphi_h q_h) \,,
    $$
    where $e_u = \tilde{u} - I_h \tilde{u}$ and $e_p = p - I_hp$.
    It now remains to estimate each of the terms. For the term $I$, using the expression \eqref{eq:coercivity_first_eq}, the trace inequality \cite[Lemma 3.5]{phiFEM2}, interpolation inequalities and the expression \eqref{eq:norm_coercitiy_dual_poisson_phi_fem}, 
    \begin{equation*}
         I  \leqslant   
         Ch^{k} \| \tilde u \|_{k+1, \Omega_h} \vertiii{(v_h, q_h)}_h\\
         \leqslant C h^k \| f \|_{k-1, \Omega}  \vertiii{(v_h, q_h)}_h\,.
    \end{equation*}
    Since $\tilde f = f $ on $\Omega$ and $\Omega \subset \Omega_h$, using Lemma \ref{lemma:phifem_neumann_lemma_34}, it holds
$$\|\tilde f - f\|_{0,\Omega_h\setminus \Omega}\leqslant Ch^{k-1}\|\tilde f - f\|_{k-1,\Omega_h\setminus \Omega}.$$    
   Hence, using Lemma \ref{lemma:phi_p} and the interpolation inequalities,
    \begin{multline*}
         II \leqslant C  \| \tilde f -f \|_{0, \Omega_h\setminus \Omega} \left( \| v_h \|_{0, \Omega_h \setminus \Omega} + \sigma h^2 \| \Delta v_h \|_{0, \Omega_h \setminus \Omega} \right) \\ 
            \leqslant C h^k \| f \|_{k-1, \Omega_h} \vertiii{(v_h, q_h)}_h\,.
    \end{multline*}
    Let us now  estimate term $III$:
   Since  $\tilde{u} = \frac{1}{h} p \varphi$ on $\Omega_h^\Gamma$,
    \begin{equation*}
        III    \leqslant \frac{C}{h} \| \tilde{u} - \frac{1}{h} \varphi_h p \|_{0, \Omega_h^\Gamma} \vertiii{(v_h, q_h)}_h  
        \leqslant \frac{C}{h}\|\varphi-\varphi_h\|_{\infty, \Omega_h^\Gamma}\|\frac{p}{h}\|_{0, \Omega_h^\Gamma}  \vertiii{(v_h, q_h)}_h\,.
        \end{equation*}
  Using Lemma \ref{lemma:hardy} and interpolation inequalities,
         \begin{align*}
         III & \leqslant C h^k \| \varphi \|_{W^\infty_{k+1}} \| \tilde{u} \|_{1, \Omega_h^\Gamma} \vertiii{(v_h, q_h)}_h \\ 
         & \leqslant C h^k \| f \|_{k-1,\Omega} \vertiii{(v_h, q_h)}_h\,.
     \end{align*}
Using the previous estimates and by definition of $\vertiii{\cdot}_h$, 
    \begin{equation*}
    | u_h - I_h \tilde u |_{1, \Omega_h} \leqslant \vertiii{(u_h - I_h \tilde{u}, p_h - I_hp)}_h \leqslant  C h^k \| f \|_{k-1, \Omega}\,.
    \end{equation*}
    Finally, using the triangle inequality and some interpolation estimates,
    \begin{equation*}
        | u_h - u |_{1, \Omega} \leqslant | u_h - I_h \tilde u |_{1, \Omega_h} + | I_h \tilde{u} - \tilde{u} |_{1, \Omega_h} \leqslant C h^k \| f \|_{k-1, \Omega} \,.
    \end{equation*}
\end{proof}

\section{Numerical illustration}

We conclude this study by numerical test cases implemented in Python, using the Dolfinx \cite{DOLFINx} library. For these two test cases, we compare our scheme  to a standard finite element method  and to the first $\varphi$-FEM scheme presented in \cite{phifem},
respectively referred as "Dual $\varphi$-FEM", "Standard-FEM" and "Direct $\varphi$-FEM" in the figures.

\subsection{First test case: a 2D complex geometry}

We first consider a complex geometry (see Fig.~\ref{fig:results_test_case_1_output}) defined by 
\begin{equation}\label{eq:phi_gaussian_reconstruction_product}
      \varphi(x,y) = (-1)^5 \prod_{j=1}^5 \bigg( -1 + \exp\left(- \frac{x_j^2}{2l_{x,j}^2} - \frac{y_j^2}{2l_{y,j}^2} \right)  \bigg) - 0.5\,,
\end{equation}
with \[ x_j = \cos(\theta_j) (x-x_{0,j}) - \sin(\theta_j) (y-y_{0,j}) \text{ and } y_j = \sin(\theta_j) (x-x_{0,j}) + \cos(\theta_j) (y-y_{0, j})  \,,\]
for some given parameters $(x_{0,j}, y_{0,j}, l_{x,j}, l_{y,j}, \theta_j)$ given in Table \ref{tab:params_phi}. This expression was chosen to fit a liver geometry in an other study.

\begin{table}
\centering
\begin{tabular}{|c|c|c|c|c|c|}
\hline
 $j$ & $x_{0,j}$& $y_{0,j}$ & $l_{x,j}$ & $l_{y,j}$ & $\theta_j$ \\
\hline
1 & 0.356 & 0.507 & 0.145 & 0.171 & 0.000 \\
2 & 0.588 & 0.589 & 0.153 & 0.090 & 0.000 \\
3 & 0.569 & 0.588 & 0.008 & 0.008 & 0.006 \\
4 & 0.308 & 0.443 & 0.055 & 0.116 & 0.622 \\
5 & 0.741 & 0.643 & 0.058 & 0.035 & 0.000 \\
\hline
\end{tabular}
\caption{\textbf{Test case 1.} Parameters of $\varphi$.}\label{tab:params_phi}
\end{table}
The source term is given by \(f(x,y)=\cos(x)\exp(y)\), and homogeneous Dirichlet boundary conditions are imposed on \(\Gamma\).
The parameters of the penalized \(\varphi\)-FEM scheme are set to \(\gamma=100\) and \(\sigma_D=0.1\).
The reference solution and the corresponding absolute errors are displayed in Fig.~\ref{fig:results_test_case_1_output}.

In Fig.~\ref{fig:results_test_case_1}, we compare the relative \(L^2\) and semi-\(H^1\) errors with respect to a fine standard FEM solution computed on a mesh with \(h \approx 0.002\).
The three methods yield comparable results, each exhibiting optimal convergence rates in both the \(L^2\) and \(H^1\) norms.

Furthermore, Fig.~\ref{fig:cond_test_case_1} (left) compares the total computational time of each method, including mesh generation, cell selection, and assembly and solution of the finite element system.
These results show that both \(\varphi\)-FEM approaches are faster than the standard FEM.
However, since the finite element system associated with the penalized scheme is larger due to the introduction of an additional variable, higher computation times are observed compared to the direct \(\varphi\)-FEM.
Finally, Fig.~\ref{fig:cond_test_case_1} (right) illustrates that optimal conditioning is achieved.

\begin{figure}
    \centering
    \includegraphics[width=\textwidth]{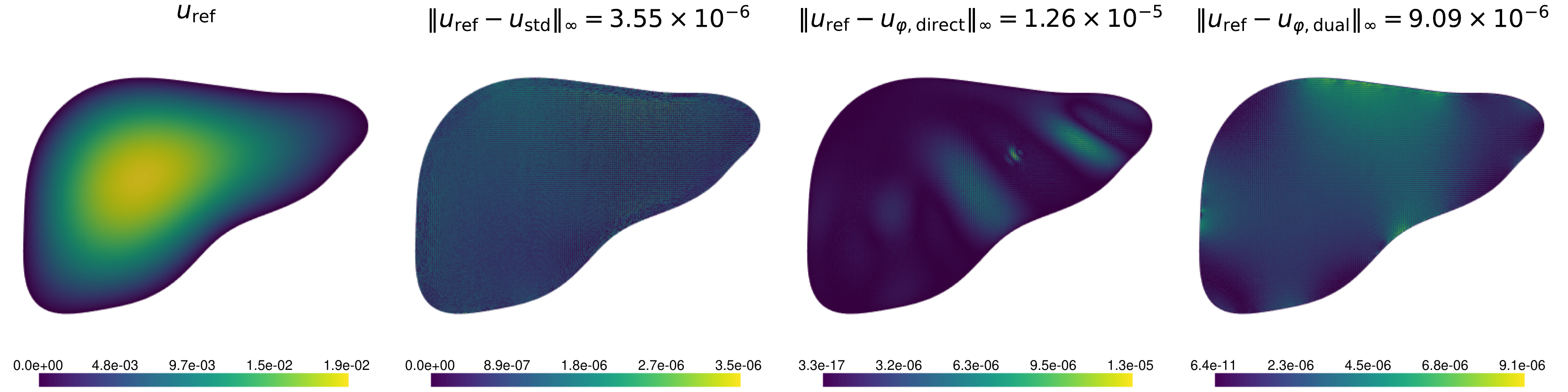}
    \caption{\textbf{Test case 1.} From left to right: reference solution (standard FEM with a fine mesh); difference between the reference solution and the projection of each approximated solution (Standard FEM, Direct $\varphi$-FEM, Dual $\varphi$-FEM).}
    \label{fig:results_test_case_1_output}
\end{figure}

\begin{figure}
    \centering
    \includegraphics[width=\textwidth]{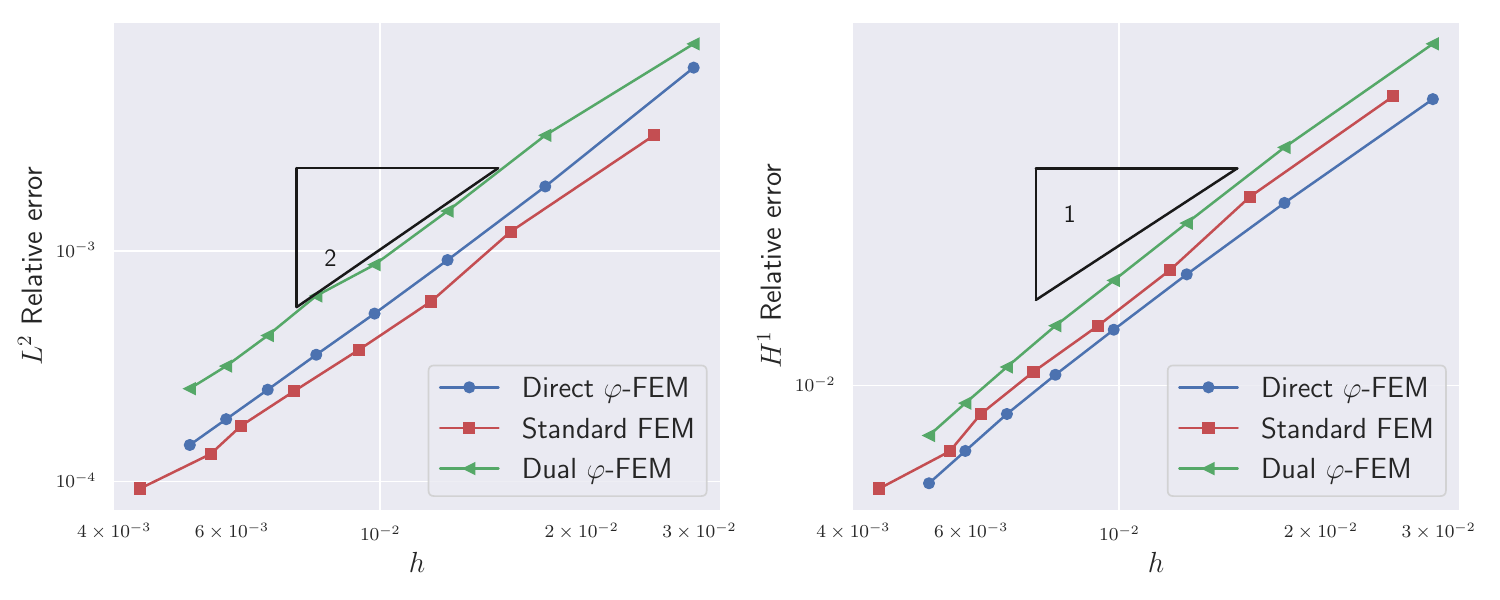}
    \caption{\textbf{Test case 1.} $L^2$ (left) and semi-$H^1$ (right) relative errors with respect to the mesh size.}
    \label{fig:results_test_case_1}
\end{figure}

\begin{figure}
    \centering
    \includegraphics[width=\textwidth]{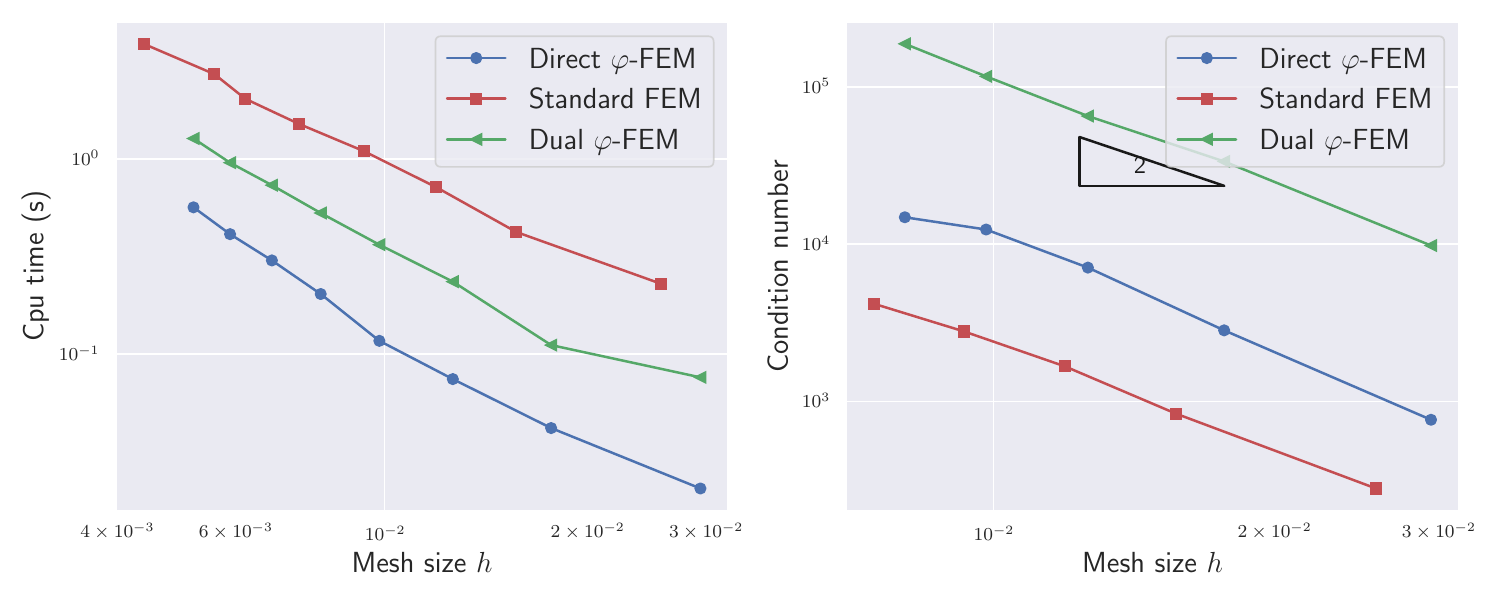}
    \caption{\textbf{Test case 1.} Left: computation time with respect to the mesh size. Right: condition number with respect to the mesh size.}
    \label{fig:cond_test_case_1}
\end{figure}

\subsection{Second test case: a 3D geometry}

For this test case, we consider a sphere described by the level-set function 
\[ \varphi(x,y,z) = -0.3125^2 + (x-0.5)^2 + (y-0.5)^2 + (z-0.5)^2, \]
and a manufactured solution given by $u_{ex} (x,y,z) = 1-\exp(\varphi(x,y,z)^2)$ such that $u_{ex}$ vanishes on $\Gamma$. {The parameters of the penalized $\varphi$-FEM scheme are set to $\gamma=100$ and $\sigma_D=0.01$.}
The reference solution and the absolute errors are given in Fig.~\ref{fig:errors_test_case_2}. The relative $L^2$ and $H^1$ errors are given in Fig~\ref{fig:results_test_case_2}, verifying the optimal convergence in both norms.

\begin{figure}
    \centering
    \includegraphics[width=\textwidth]{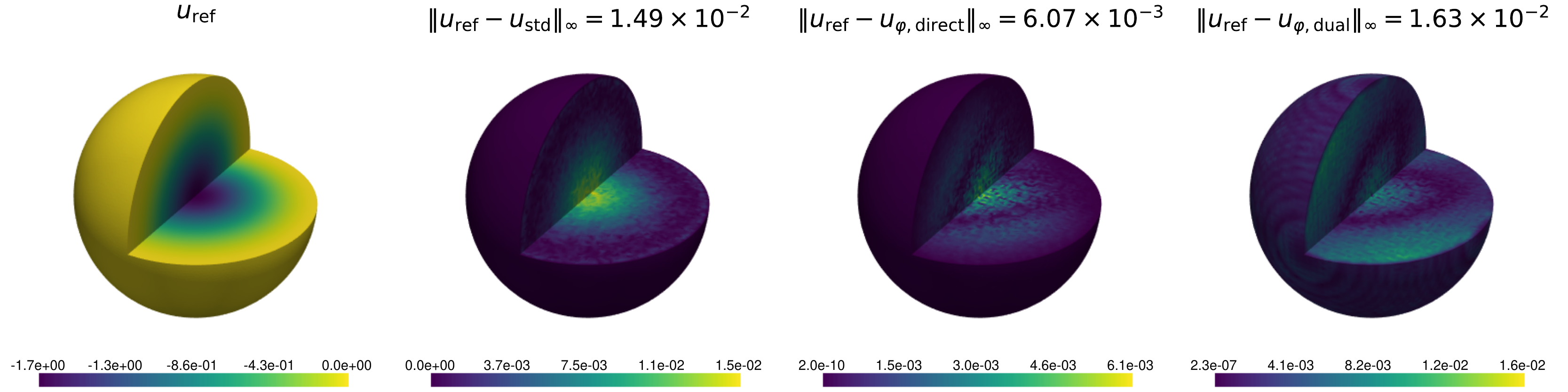}
   \caption{\textbf{Test case 2.} From left to right: reference solution (standard FEM with a fine mesh); difference between the reference solution and the projection of each approximated solution (Standard FEM, Direct $\varphi$-FEM, Dual $\varphi$-FEM).}   
   \label{fig:errors_test_case_2}
\end{figure}

\begin{figure}
    \centering
    \includegraphics[width=\textwidth]{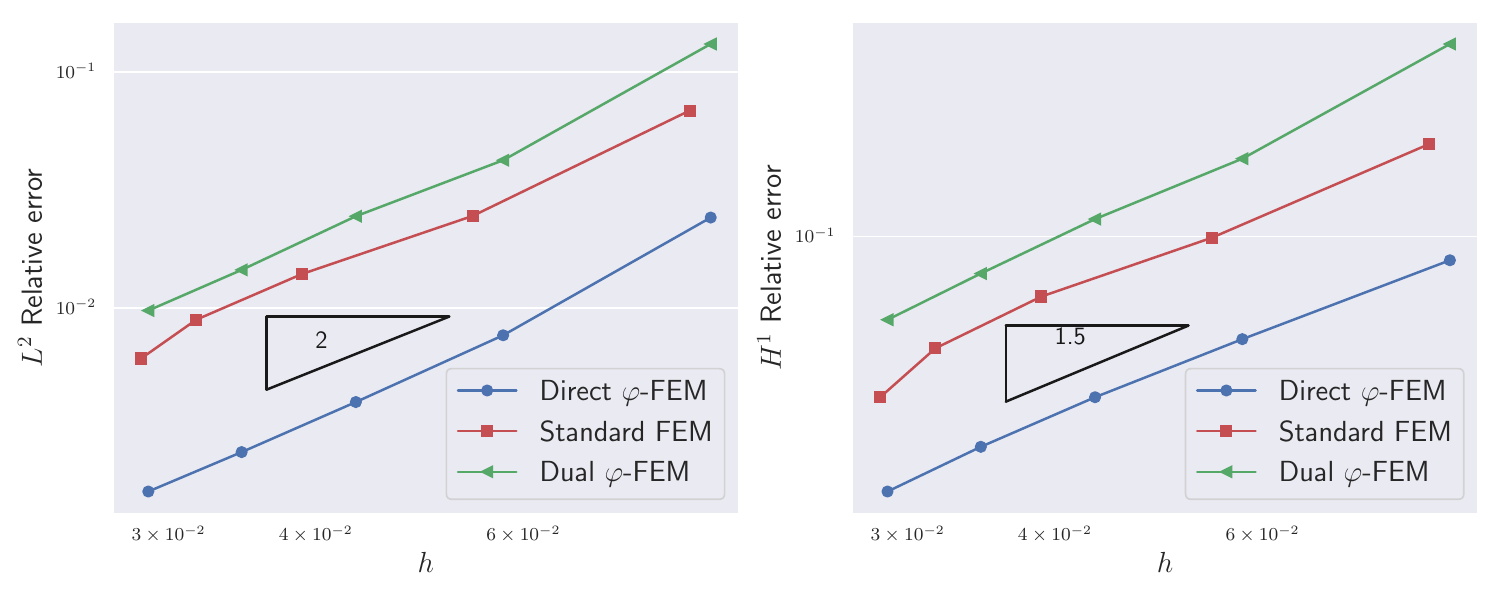}
   \caption{\textbf{Test case 2.} $L^2$ (left) and semi-$H^1$ (right) relative errors with respect to the mesh size.}   
   \label{fig:results_test_case_2}
\end{figure}

\section*{Acknowledgment}

This work was supported by the Agence Nationale de la Recherche, Project PhiFEM, under grant ANR-22-CE46-0003-01.
%\section{Conclusion}
%
%TO DO

\bibliographystyle{plain}
\bibliography{biblio}
\end{document}